\documentclass[10pt]{article}
\usepackage{amssymb, amsfonts, amsthm, mathrsfs}
\usepackage[super]{nth}
\usepackage[pagebackref,colorlinks=true]{hyperref}
\usepackage{graphicx, xcolor}
\usepackage[all]{xy}
\usepackage[leqno]{amsmath}
\usepackage{titlesec}

\theoremstyle{definition}
\newtheorem{thm}{Theorem}[section]

\newtheorem{cor}[thm]{Corollary}
\newtheorem{prop}[thm]{Proposition}

\newtheorem{rmk}[thm]{Remark}

\makeatletter

\def\tr{\mathrm{Tr}}

\def\ccc{\mathbb{C}}

\def\zz{\mathbb{Z}}
\def\rr{\mathbb{R}}

\def\pp{\mathbb{P}}

\def\pt{\partial}

\def\bpt{\bar{\pt}}

\def\ud{\mathrm{d}}

\def\bz{\bar{z}}

\def\ric{\mathrm{Ric}}

\makeatother

\begin{document}

\title{\textbf{Anomaly Flow and T-Duality}}

\author{Teng Fei and Sebastien Picard}

\date{}

\maketitle{}

\abstract{In this paper, we study the dual Anomaly flow, which is a dual version of the Anomaly flow under T-duality. A family of monotone functionals is introduced and used to estimate the dilaton function along the flow. Many examples and reductions of the dual Anomaly flow are worked out in detail.}

\section{Introduction}

Consider the following equation on a compact complex Calabi-Yau $n$-fold $(X^n,\Omega)$:
\begin{equation}\label{daf}
\pt_t\left(\|\Omega\|_\omega\omega^{n-1}\right)=-i\pt\bpt\left(\|\Omega\|^{\frac{2(n-2)}{n-1}}_\omega\cdot\omega^{n-2}\right),
\end{equation}
where $\omega=\omega_t$ is an evolving family of Hermitian metrics. Our definition of a Calabi-Yau manifold $(X,\Omega)$ is a complex manifold $X$ equipped with a holomorphic volume form $\Omega$ trivializing the canonical bundle $K_X$ of $X$. We will always assume that the initial metric $\omega_0$ is conformally balanced \cite{li2005, michelsohn1982}, i.e.

\begin{equation}\label{confbal}
\ud\left(\|\Omega\|_{\omega_0}\omega_0^{n-1}\right)=0.
\end{equation}

It is clear from the evolution equation (\ref{daf}) that the conformally balanced condition is preserved along the flow. Calabi-Yau manifolds admitting conformally balanced metrics are of interest in both heterotic string theory and non-K\"ahler complex geometry; see e.g. \cite{fu2010, fu2012, fu2014, goldstein2004, hull1986b, martelli2011, ossa2014, tosatti2017b, tseng2012, strominger1986}. We will discuss several examples of these manifolds in \S \ref{section-examples}. In \S \ref{section-evol-eqn}, we will see that equation (\ref{daf}) can be rewritten as an evolution equation of the Hermitian metric $\omega_t$.

We shall call equation (\ref{daf}) the \emph{dual Anomaly flow}, as it can be derived formally by taking T-duality of the Anomaly flow. The Anomaly flow on a threefold $(X^3,\Omega)$ is given by
\begin{equation}
\pt_t\left(\|\Omega\|_\omega\omega^{2}\right)=i\pt\bpt \omega - {\alpha' \over 4} ( {\rm Tr} \, Rm \wedge Rm - \Phi(t)),
\end{equation}
where $\alpha' \in \mathbb{R}$ is a fixed parameter, and $\Phi(t)$ is a given path of closed $(2,2)$ forms. For many examples of interest, $\Phi = {\rm Tr} \, F \wedge F$, where $F$ is the curvature of a connection on a holomorphic vector bundle $E \rightarrow X$.

The Anomaly flow was introduced by Phong-Picard-Zhang \cite{phong2018b,phong2018c,phong2017} as a parabolic system incorporating the anomaly cancellation equation of the Hull-Strominger system \cite{hull1986b,hull1986c,strominger1986} and its higher dimension analogue \cite{fu2008}. Its primary motivation was to implement the conformally balanced condition in the absence of an analogue of the $\partial \bar{\partial}$-lemma. The Anomaly flow has been shown to be a powerful tool in the study of the Hull-Strominger system and in non-K\"ahler complex geometry; see for example \cite{phong2018d,phong2018e,phong2017b,fei2017d,phong2018c}.

The Anomaly flow can also be studied while setting the $\alpha'$-corrections to zero, which on an $n$-fold $(X^n,\Omega)$ is given \cite{phong2018e,phong2018c} by

\begin{equation}\label{af}
\pt_t\left(\|\Omega\|_\omega\omega^{n-1}\right)=i\pt\bpt\left(\omega^{n-2}\right).
\end{equation}

We will show in \S \ref{section-sfcy} that for semi-flat Calabi-Yau $n$-folds, the dual Anomaly flow (\ref{daf}) is induced by (\ref{af}) on the dual torus fibration via T-duality.

T-duality is a duality in theoretical physics which relates the Type IIA and Type IIB string theory through mirror symmetry \cite{strominger1996}. It is unclear what should be the right notion of mirror symmetry for heterotic string theories which include the compactifications described by the Hull-Strominger system. As a modest attempt, we investigate the behavior of the Anomaly flow under T-duality to probe the geometry of the mirror Hull-Strominger system. However, our results are far from conclusive as we still do not know how to incorporate the $\alpha'$-correction terms into our considerations. Other proposed frameworks for studying T-duality and the Hull-Strominger system include \cite{baraglia2015,garciafernandez2016,garciafernandez2018b,lau2015}.


Compared to the Anomaly flow, the first striking difference is that the dual Anomaly flow is not a parabolic flow in the usual sense, and therefore we do not know whether a general short-time existence statement will hold. However, we shall see that there are many reductions of the dual Anomaly flow which are parabolic in nature, including the cases of generalized Calabi-Gray manifolds, K\"ahler Calabi-Yau manifolds, and other examples discussed in \S \ref{section-examples}. This suggests that there is a large class of initial data where the flow is indeed parabolic. Furthermore, we demonstrate that the dual Anomaly flow can be viewed as a generalization of the inverse Monge-Amp\`ere flow recently introduced by Cao-Keller \cite{cao2013} and Collins-Hisamoto-Takahashi \cite{collins2017}.

As is the case for the Anomaly flow (\ref{af}) with $\alpha'=0$, the dual Anomaly flow can be used to produce a continuous deformation path from conformally balanced metrics to K\"ahler metrics. Indeed, as shown in Theorem \ref{thm-station}, the stationary points of the flow are K\"ahler Ricci-flat metrics. Since the flow preserves the balanced class of the initial metric, the flow seeks a K\"ahler representative in a given balanced class. This is a subtle problem, as not all balanced classes on K\"ahler manifolds admit a K\"ahler metric; see \cite{tosatti2009,fu2014} for counterexamples and \cite{collins2017b,lejmi2015,xiao2016} for related conjectures. We hope that the dynamics of the dual Anomaly flow and Anomaly flow can help us understand when a balanced class admits a K\"ahler metric.

As a step in the development of the analysis of these evolution equations, we find a family of functionals which produce monotone quantities along both the dual Anomaly flow (\ref{daf}) and the Anomaly flow (\ref{af}). These functionals, parameterized by $\alpha \in \mathbb{R}$, are given by
\[
\mathcal{F}_\alpha(\omega) = \int_X \| \Omega \|_\omega^\alpha \, \frac{\omega^n}{n!}.
\]
When $\alpha=1$, this functional was used in \cite{garciafernandez2018} to formulate a variational principle for the Hull-Strominger system, and also used in \cite{phong2018c} to obtain convergence of the Anomaly flow with conformally K\"ahler initial data. For negative $\alpha$, this functional is monotone along the dual Anomaly flow, while for large positive $\alpha$, this functional is monotone along the Anomaly flow. Using these monotone quantities, we obtain the following estimates.

\begin{thm} \label{thm-estimate}
Let $(X,\Omega)$ be a compact Calabi-Yau manifold.
\smallskip
\par \noindent (a) Suppose $\check{\omega}_t$ solves the Anomaly flow (\ref{af}) on $X \times [0,T]$ with conformally balanced initial data $\check{\omega}_0$. Then
\[
\| \Omega \|_{\check{\omega}_t}(x,t) \leq \sup_X \| \Omega \|_{\check{\omega}_0}
\]
for all $(x,t) \in X \times[0,T]$.
\smallskip
\par \noindent (b) Suppose $\hat{\omega}_t$ solves the dual Anomaly flow (\ref{daf}) on $X \times [0,T]$ with conformally balanced initial data $\hat{\omega}_0$. Then
\[
\| \Omega \|_{\hat{\omega}_t}(x,t) \geq \inf_X \| \Omega \|_{\hat{\omega}_0}
\]
for all $(x,t) \in X \times[0,T]$.
\end{thm}

As both these flows deform conformally balanced metrics towards K\"ahler metrics, they must diverge on manifolds which do not admit a K\"ahler metric. For known non-K\"ahler examples, the norm $\| \Omega \|_\omega$ shrinks to zero along the Anomaly flow, while the norm $\| \Omega \|_\omega$ expands to infinity along the dual Anomaly flow.

This paper is organized as follows. In Section \ref{section-sfcy}, we review the classical construction of semi-flat Calabi-Yau manifolds and then use them to derive the dual Anomaly flow equation (\ref{daf}). Basic evolution equations and the geometry of stationary points are derived in Section \ref{section-evol-eqn}. In Section \ref{section-monotone}, we prove monotonicity of the functional $\mathcal{F}_\alpha$ and prove Theorem \ref{thm-estimate}. In Section \ref{section-examples}, we investigate the dual Anomaly flow on various Calabi-Yau manifolds. In particular we show that the dual Anomaly flow with conformally K\"ahler initial data is equivalent to the inverse Monge-Amp\`ere flow, and thus the flow exists for all time and converges. However, for non-K\"ahler examples, including the generalized Calabi-Yau manifolds, the Iwasawa manifold and quotients of $SL(2,\ccc)$, various kinds of singularities may occur, demonstrating the rich and diverse behavior of the dual Anomaly flow.
~\\

\noindent\textbf{Acknowledgments}\\
We would like to thank S.-T. Yau and X.-W. Zhang for helpful discussions.

\section{Semi-Flat Calabi-Yau Manifolds and the Dual Anomaly Flow} \label{section-sfcy}

\subsection{Semi-flat Calabi-Yau manifolds}

First, let us review the construction of semi-flat Calabi-Yau manifolds. More details can be found in \cite{leung2005}, among other sources.

Let $B$ be an $n$-dimensional compact special integral affine manifold, in the sense that we can cover $B$ by local coordinate charts $\{x^1,x^2,\dots,x^n\}$ such that the transition functions are valued in the affine group $SL(n,\zz)\ltimes\rr^n$. Let $TB$ be the tangent bundle of $B$, with $\{y^1,y^2,\dots,y^n\}$ the natural coordinates on the tangent directions. It is clear that the fiberwise lattice
\[\check{\Lambda}=\zz\frac{\pt}{\pt x^1}+\dots+\zz\frac{\pt}{\pt x^n}=\{(y^1,\dots,y^n):y^j\in\zz\}\]
is well-defined. We can thus form the fiberwise quotient $\check{X}:=TB/\check{\Lambda}$, which has a natural $T^n$-fibration structure over $B$.

In addition, $\check{X}$ has a complex structure with holomorphic coordinates $z^j=x^j+iy^j$ for $j=1,\dots,n$. Furthermore, $\check{X}$ is a Calabi-Yau manifold as it admits a well-defined holomorphic volume form
\[\check{\Omega}=\ud z^1\wedge\dots\wedge\ud z^n.\]

For any Riemannian metric $g$ on $B$, there is a natural Hermitian metric on $\check{X}$ making $\check{\pi}:\check{X}\to B$ a Riemannian submersion. Its associated $(1,1)$-form is given by
\[\check{\omega}=\frac{i}{2}g_{jk}\ud z^j\wedge\ud\bz^k.\]
This is the so-called semi-flat metric, as it is $T^n$-translational invariant. Moreover, we have
\[\|\check{\Omega}\|^{-2}_g=c\cdot\det(g_{ab}),\]
where $c$ is a positive constant which can be taken to be 1 without loss of generality.

It is also clear that the metric $\check{\omega}$ is K\"ahler if and only if $g$ is a Hessian metric, i.e. in local coordinates one can find a strictly convex function $\phi$ such that
\[g_{jk}=\frac{\pt^2\phi}{\pt x^j\pt x^k}.\]
In summary, given a special integral affine manifold $B$ equipped with a Hessian metric, one can construct a Calabi-Yau manifold $\check{X}$ with a K\"ahler metric $\check{\omega}$. The metric $\check{\omega}$ is Ricci-flat if the potential $\phi$ solves the real Monge-Amp\`ere equation
\[\det(\nabla^2\phi)=1.\]

In this case, the natural projection $\check{X}\to B$ defines a special Lagrangian fibration.

Given a Hessian metric $g$, we can cook up a new set of affine coordinates $\left\{x_j=\dfrac{\pt\phi}{\pt x^j}\right\}$ by taking the Legendre transform on $B$, so that
\[\frac{\pt x_j}{\pt x^k}=g_{jk}\]
is a positive definite matrix. We write $(g^{jk})=(g_{jk})^{-1}$, so that the Hessian metric $g$ can also be written as
\[g=g^{jk}\ud x_j\otimes\ud x_k\]
on $B$. Again, there exists a local potential $\psi=\sum_jx^jx_j-\phi$ such that
\[g^{jk}=\frac{\pt^2\psi}{\pt x_j\pt x_k}.\]

Let $\{y_1,\dots,y_n\}$ be the natural coordinates on the cotangent directions (with respect to the coordinates $\{x^1,\dots,x^n\}$). It is clear that the fiberwise lattice
\[\hat{\Lambda}=\zz\ud x^1+\dots+\zz\ud x^n=\{(y_1,\dots,y_n):y_k\in\zz\}\]
is well-defined and we may form the fiberwise quotient
\[\hat{X}:=T^*B/\hat{\Lambda},\] which is also a $T^n$-fibration over $B$.

Moreover, $z_k=x_k+iy_k$ with $k=1,\dots,n$ is a set of holomorphic coordinates on $\hat{X}$ and the natural Hermitian metric on $\hat{X}$ can be expressed as
\[\hat{\omega}=\frac{i}{2}g^{jk}\ud z_j\wedge\ud\bz_k.\]

\begin{rmk}
The metric $g$ gives a natural identification of $TB$ with $T^*B$, explicitly given by
\[x^j=x^j,\quad y_l=y^k\phi_{kl}(x),\]
however under this identification the complex structures on $TB$ and $T^*B$ are not the same.
\end{rmk}
\begin{rmk}
The symplectic form $\hat{\omega}$ can be rewritten as
\[\hat{\omega}=\sum_j\ud x^j\wedge\ud y_j,\]
which is exactly induced from the canonical symplectic form on $T^*B$. In other words, the complex structure on $\hat{X}$ depends on the choice of the Hessian metric $g$, however the K\"ahler form $\hat{\omega}$ is independent of this choice. In this sense there is a correspondence between the K\"ahler moduli of $\check{X}$ and the complex moduli of $\hat{X}$, which provides the semi-flat model of mirror symmetry.
\end{rmk}

In conclusion, from a special integral affine manifold with a Hessian metric, one can construct a pair of Calabi-Yau manifolds $\check{X}$ and $\hat{X}$ with K\"ahler metrics. The passage from $(\check{X},\check{\omega})$ to $(\hat{X},\hat{\omega})$ is known as the T-duality, which manifests the case of mirror symmetry without quantum corrections.

\subsection{Toy model: the K\"ahler-Ricci flow}

Let us consider the K\"ahler-Ricci flow
\[\pt_t\check{\omega}=-\ric(\check{\omega})=i\pt\bpt\log(\check{\omega}^n)\]
on $\check{X}$. In the semi-flat setting, the K\"ahler-Ricci flow reduces to the following equation on $B$
\[\pt_t(g_{jk})=\frac{1}{2}\frac{\pt^2}{\pt x^j\pt x^k}\log\det(g_{ab}).\]
In terms of the dual coordinates $\{x_1,\dots,x_n\}$, we notice that the above equation can be expressed as
\[\pt_t(g^{jk})=\frac{1}{2}\frac{\pt^2}{\pt x_j\pt x_k}\log\det(g^{ab})-\frac{1}{2}\sum_q\frac{\pt}{\pt x^q}\log\det(g^{ab})\frac{\pt}{\pt x_j}(g^{qk}),\]
whose leading order term is exactly the semi-flat reduction of the K\"ahler-Ricci flow on $\hat{X}$.

By running the K\"ahler-Ricci flow on $\check{X}$, the complex structure is fixed and the flow is a flow of K\"ahler metrics. Correspondingly, on the mirror side $\hat{X}$, the symplectic structure on $\hat{X}$ is fixed, however its complex structure evolves along time. If we think of the first order terms in the above expression as coming from correction of complex structure, then we may conclude that the K\"ahler-Ricci flow is self-dual under T-duality, as we only care about the leading order term. In next subsection, we will adopt this point of view of discarding first order terms to derive the dual version of the Anomaly flow under T-duality.

\subsection{The dual Anomaly flow}

Let us run the Anomaly flow (\ref{af}) on $\check{X}$. As in \cite{phong2018c}, we assume that the initial metric $\check{\omega}'_0$ is conformally K\"ahler
\[\check{\omega}'_0=\|\check{\Omega}\|_{\check{\omega}_0}^{-\frac{2}{n-2}}\check{\omega}_0,\]
where $\check{\omega}_0$ is a semi-flat K\"ahler metric on $\check{X}$ constructed as before. In this case, the Anomaly flow preserves the conformally K\"ahler condition and after a conformal change the Anomaly flow can be written as a flow of K\"ahler metrics
\[\pt_t\check{\omega}=\frac{i}{n-1}\pt\bpt(\|\check{\Omega}\|_{\check{\omega}}^{-2}).\]
Furthermore this flow reduces to a flow on $B$
\[\pt_t(g_{jk})=\frac{1}{2(n-1)}\frac{\pt^2}{\pt x^j\pt x^k}\det(g_{ab}).\]
Using the argument in the previous subsection, it is not hard to see that the dual version is
\[\pt_t(g^{jk})=-\frac{1}{2(n-1)}\frac{\pt^2}{\pt x_j\pt x_k}(\det(g^{ab})^{-1}),\]
or equivalently
\[\pt_t\hat{\omega}=-\frac{i}{n-1}\pt\bpt(\|\hat{\Omega}\|^2_{\hat{\omega}}).\]
Since T-duality maps a circle of radius $R$ to a circle of radius of $R^{-1}$, we conclude that the dual conformally K\"ahler metric $\hat{\omega}'$ is given by
\[\hat{\omega}'=\|\check{\Omega}\|_{\check{\omega}}^{\frac{2}{n-2}}\hat{\omega}=\|\hat{\Omega}\|_{\hat{\omega}}^{-\frac{2}{n-2}}\hat{\omega}.\]
In terms of the dual metric $\hat{\omega}'$, the dual Anomaly flow reads
\[\pt_t(\|\hat{\Omega}\|_{\hat{\omega}'}\cdot\hat{\omega}'^{n-1})=-i\pt\bpt\left(\|\hat{\Omega}\|_{\hat{\omega}'}^{\frac{2(n-2)}{n-1}}\hat{\omega}'^{n-2} \right),\]
which is exactly equation (\ref{daf}). For this reason, we shall call (\ref{daf}) the \emph{dual Anomaly flow}, which makes sense for any Calabi-Yau manifold (not necessarily K\"ahler). However, we only restrict ourselves to the case that the initial metric satisfies the conformally balanced equation (\ref{confbal}).

\section{Evolution Equations and Stationary Points} \label{section-evol-eqn}

\subsection{Evolution equations}

We now derive various evolution equations along the dual Anomaly flow
\[\pt_t\left(\|\Omega\|_\omega\cdot\omega^{n-1}\right)=-i\pt\bpt\left(\|\Omega\|^{\frac{2(n-2)}{n-1}}_\omega\cdot\omega^{n-2}\right)\]
with conformally balanced initial data.

First, just like the case of Anomaly flow \cite{phong2018e}, the dual Anomaly flow is actually a flow of Hermitian metrics. To see this, let us denote the right hand side of (\ref{daf}) by $\Phi=-i\pt\bpt\left(\|\Omega\|^{\frac{2(n-2)}{n-1}}_\omega\cdot\omega^{n-2}\right)$. The flow can be rewritten as
\[-\frac{1}{2}\|\Omega\|_\omega\tr~\dot\omega\cdot\omega^{n-1}+(n-1)\|\Omega\|_\omega\dot\omega\wedge\omega^{n-2}=\Phi.\]
Here we defined the trace of a $(1,1)$ form $\alpha$ by
\[
n \alpha \wedge \omega^{n-1} = (\tr~\alpha) \, \omega^n.
\]
By taking the Hodge star of both sides, we get
\[*\Phi=(n-1)!\|\Omega\|_\omega\left(\frac{1}{2}\tr~\dot\omega\cdot\omega-\dot\omega\right).\]
Taking the trace of both sides, we get
\[(n-1)!\|\Omega\|_\omega\tr~\dot\omega=\frac{2\tr(*\Phi)}{n-2},\]
and therefore
\begin{equation}
 \label{daf2}
(n-1)!\|\Omega\|_\omega\dot\omega=-*\Phi+\frac{\tr(*\Phi)}{n-2}\omega.
\end{equation}
Hence the dual Anomaly flow is indeed equivalent to a flow of Hermitian metrics.

To get an explicit expression for the flow of $\omega$, we need to compute $*\Phi$. 
It is straightforward to calculate that
\[\begin{split}\Phi=&-(n-2)\|\Omega\|_\omega^{\frac{2(n-2)}{n-1}}\left(\frac{1}{n-1}\rho\wedge\omega^{n-2}+\frac{4(n-2)}{(n-1)^2}i\theta\wedge\bar\theta\wedge\omega^{n-2}\right.\\ &+\left.\frac{2(n-2)}{n-1}(T\wedge\bar\theta+\bar T\wedge\theta)\wedge\omega^{n-3}+i\pt\bpt\omega\wedge\omega^{n-3}+i(n-3)T\wedge\bar T\wedge\omega^{n-4}\right),\end{split}\]
where
\[ \theta = \partial \log \| \Omega \|_\omega, \ \ T = i \partial \omega, \ \ \bar{T} = - i \bar{\partial} \omega\]
are certain differential forms involving first derivatives of the metric, and
\[ \rho=2i\pt\bpt\log\|\Omega\|_\omega \]
is the (first) Chern-Ricci form. For simplicity, we write
\[\Phi=-(n-2)\|\Omega\|_\omega^{\frac{2(n-2)}{n-1}}(A\wedge\omega^{n-2}+B\wedge\omega^{n-3}+(n-3)C\wedge\omega^{n-4})\]
for
\[\begin{split}A&=\frac{\rho}{n-1}+\frac{4i(n-2)}{(n-1)^2}\theta\wedge\bar\theta,\\ B&=\frac{2n-4}{n-1}(T\wedge\bar\theta+\bar T\wedge\theta)+i\pt\bpt\omega,\\ C&=iT\wedge\bar{T}.\end{split}\]

By using the primitive decomposition of differential forms and the Hodge star operator, it is possible to compute that the dual Anomaly flow (\ref{daf2}) is equivalent to
\begin{equation}\label{half}
\dot\omega=-\frac{\|\Omega\|_\omega^{\frac{n-3}{n-1}}}{n-1}\left((n-2)A+\Lambda B+\frac{\Lambda^2C}{2}+\left(\Lambda A-\frac{\Lambda^3C}{6(n-2)}\right)\omega\right),
\end{equation}
where $\Lambda$ is the operator of contracting with $\omega$, and we should take $C=0$ when $n=3$.

To demonstrate how this works, we give the details of the calculation of $*(C\wedge\omega^{n-4})$, which is one term in $*\Phi$. Notice that $C$ is a $(3,3)$-form. Let
\[C=P^6+P^4\wedge\omega+P^2\wedge\omega^2+P^0\omega^3\]
be the primitive decomposition of $C$, where $P^k$ is a primitive $k$-form.
It follows that
\[C\wedge\omega^{n-4}=P^2\wedge\omega^{n-2}+P^0\omega^{n-1}.\]
Using the well-known Hodge star operator formula for primitive forms (see for example \cite[p. 37]{huybrechts2005}), we get
\[*(C\wedge\omega^{n-4})=-(n-2)!P^2+(n-1)!P^0\omega.\]
Applying $\Lambda$ to $C$ repeatedly and using the fact that $\{H,L,\Lambda\}$ forms an $sl_2$ triple (using notation of \cite{huybrechts2005}), we derive that
\[\begin{split}\Lambda C&=(n-2)P^4+2(n-3)P^2\wedge\omega+3(n-2)P^0\omega^2,\\ \frac{\Lambda^2C}{2}&=(n-3)(n-2)P^2+3(n-2)(n-1)P^0\omega,\\ \frac{\Lambda^3C}{6}&=(n-2)(n-1)nP^0.\end{split}\]
Therefore we conclude that
\[\begin{split}*(C\wedge\omega^{n-4})&=(n-4)!\left(\frac{\Lambda^3C}{6}\omega-\frac{\Lambda^2C}{2}\right),\\ \tr*(C\wedge\omega^{n-4})&=(n-3)!\frac{\Lambda^3C}{6}.\end{split}\]
Other terms can be calculated in a similar manner, leading to
\[\begin{split}*(A \wedge\omega^{n-2})&= - (n-2)! A + (n-2)! \Lambda A \, \omega,\\ * (B \wedge\omega^{n-3})&= -(n-3)!\Lambda B + (n-3)! \frac{\Lambda^2 B}{2} \omega.\end{split}\]

To obtain an explicit expression in (\ref{half}), let us make use of the following identities in conformally balanced geometry \cite{phong2018e}
\[\rho'=\rho''=\frac{1}{2}\rho,\quad\theta=i\bpt^*\omega=:\tau,\]
where
\[\rho'=iR_{\bar{k}j}'\ud z^j\wedge\ud\bar{z}^k,\quad \rho''=iR_{\bar{k}j}''\ud z^j\wedge\ud\bar{z}^k\]
in the notation of \cite{phong2018e}. The forms $\rho'$ and $\rho''$ are sometimes known as the 3rd and 4th Chern-Ricci forms. Additionally in local coordinates, we have the following expressions
\[\omega=ig_{\bar{k}j}\ud z^j\wedge\ud\bz^k,\quad T=\frac{1}{2}T_{\bar{k}jl}\ud z^l\wedge\ud z^j\wedge\ud\bz^k,\quad \tau=g^{j\bar{k}}T_{\bar{k}jl}\ud z^l,\]
and the curvature conventions
\[
R_{\bar{k} j}{}^p{}_q = -\partial_{\bar{k}} (g^{p \bar{s}} \partial_j g_{\bar{s} q}),
\]
\[
R_{\bar{k} j} = R_{\bar{k} j}{}^p{}_p, \ \ \tilde{R}_{\bar{k} j} = R^p{}_{p \bar{k} j}, \ \ R_{\bar{k} j}' = R_{\bar{k}}{}^p{}_{pj}, \ \ R_{\bar{k} j}'' = R^p{}_{j \bar{k} p}.
\]
Straightforward calculation reveals that
\[\begin{split}A&=\frac{\rho}{n-1}+\frac{4i(n-2)}{(n-1)^2}\tau\wedge\bar{\tau},\\ \Lambda A&=\frac{R}{n-1}+\frac{4(n-2)}{(n-1)^2}\|\tau\|^2,\\ \Lambda B&=-\tilde{\rho}+iT\boxdot\bar T+\frac{2n-4}{n-1}\Lambda(\pt\omega\wedge\pt^*\omega+\bpt\omega\wedge\bpt^*\omega),\\ \frac{\Lambda^2C}{2}&=-\Lambda(\pt\omega\wedge\pt^*\omega+\bpt\omega\wedge\bpt^*\omega)-i\tau\wedge\bar\tau-iT\boxdot\bar T-\frac{i}{2}T\circ\bar T,\\ \frac{\Lambda^3C}{6}&=\|\tau\|^2-\frac{1}{2}\|T\|^2.\end{split}\]
Here $\tilde{\rho}=i\tilde{R}_{\bar{k}j}\ud z^j\wedge\ud\bar{z}^k $, and $iT\boxdot\bar{T}$ and $iT\circ\bar{T}$ are nonnegative real (1,1)-forms as in \cite{liu2017} defined by
\[\begin{split}iT\boxdot\bar{T}&=iT_{\bar{k}st}\overline{T_{\bar{j}uv}}g^{j\bar{k}}g^{t\bar{v}}\ud z^s\wedge\ud\bz^u,\\ iT\circ\bar{T}&=iT_{\bar{k}st}\overline{T_{\bar{j}uv}}g^{s\bar{u}}g^{t\bar{v}}\ud z^j\wedge\ud\bz^k.
\end{split}\]
They satisfy
\[\Lambda(iT\boxdot\bar{T})=\Lambda(iT\circ\bar{T})=\|T\|^2\geq0.\]

Combining all our calculations, we get the evolution equation of the metric
\begin{equation}\label{evo}
\begin{split}\dot\omega=&-\frac{\|\Omega\|_\omega^{\frac{n-3}{n-1}}}{n-1}\left(\frac{n-2}{n-1}\rho-\tilde{\rho}+\frac{R\omega}{n-1}\right.\\ +& \frac{(n-3)(3n-5)}{(n-1)^2}i\tau\wedge\bar\tau+\frac{n-3}{n-1}\Lambda(\pt\omega\wedge\pt^*\omega+\bpt\omega\wedge\bpt^*\omega)-\frac{i}{2}T\circ\bar T\\ +&\left.\left(\frac{(n-3)(3n-5)}{(n-1)^2(n-2)}\|\tau\|^2+\frac{\|T\|^2}{2n-4}\right)\omega\right)\end{split}
\end{equation}
for $n>3$
and
\[\dot\omega=-\frac{1}{4}\left(\rho-2\tilde{\rho}+2i\tau\wedge\bar\tau+2\Lambda(\pt\omega\wedge\pt^*\omega+\bpt\omega\wedge\bpt^*\omega)+ 2iT\boxdot\bar{T}+R\omega+2 \|\tau\|^2\omega\right)\]
for $n=3$.

The reason that we need a separate formula for $n=3$ case is that in deriving (\ref{evo}) we need to cancel certain $(n-3)$ factors. However, if we put $n=3$ in (\ref{evo}), we do get the right evolution equation for $\omega$
\begin{equation}\label{evo3}
\dot\omega=-\frac{1}{4}\left(\rho-2\tilde{\rho}+R\omega-iT\circ\bar{T}+\|T\|^2\omega\right).
\end{equation}

It is because that the following identity holds when $n=3$.
\begin{prop}~\\
If $X$ is of complex dimension 3, then for any Hermitian metric $\omega$, we have the following identity
\[\Lambda(\pt^*\omega\wedge\pt\omega+\bpt^*\omega\wedge\bpt\omega)=iT\boxdot\bar T+\frac{i}{2}T\circ\bar{T}-\frac{\|T\|^2}{2}\omega+\|\tau\|^2\omega+i\tau\wedge\bar\tau.\]
\end{prop}
\begin{proof}
Straightforward calculation.
\end{proof}
\begin{rmk}~\\
Notice that all the terms in the above identity are real (1,1)-forms which are quadratic in first derivatives of the metric $\omega$. However it seems that there is no analogous identities in higher dimensions.
\end{rmk}

The flow of the metric $\omega$ (\ref{evo}) is not parabolic, so it is unclear whether a general short-time existence statement will hold. It is intriguing that all examples studied so far do in fact reduce to an equation which admits a solution for a short-time; see \S \ref{section-examples}.


Lastly, we note that a straight-forward calculation gives the following evolution equations
\begin{eqnarray}
\label{id1}\dot\omega\wedge\frac{\omega^{n-1}}{(n-1)!}&=&-\frac{\|\Omega\|_\omega^{\frac{n-3}{n-1}}}{n-1}\left(R+\frac{2n-6}{n-2}\|\tau\|^2+\frac{1}{n-2}\|T\|^2\right)\frac{\omega^n}{n!},\\ \label{id2}\pt_t\|\Omega\|_\omega&=&\frac{\|\Omega\|_\omega^{\frac{2n-4}{n-1}}}{2(n-1)}\left(R+\frac{2n-6}{n-2}\|\tau\|^2+\frac{1}{n-2}\|T\|^2\right).
\end{eqnarray}

\subsection{Stationary points}

Our next result implies that stationary points of the dual Anomaly flow must be K\"ahler Ricci-flat. The dual Anomaly flow and the Anomaly flow are thus two different continuous deformations from conformally balanced complex geometry to K\"ahler geometry.

\begin{thm} \label{thm-station}~\\
Let $(X^n,\Omega)$ be a compact Calabi-Yau manifold of dimension $n \geq 3$. A Hermitian metric $\omega$ satisfying
\[
d ( \| \Omega \|_\omega \omega^{n-1}) = 0,
\]
and
\[
i \partial \bar{\partial} ( \| \Omega \|_\omega^{2(n-2) \over n-1} \omega^{n-2}) =0,
\]
must be K\"ahler Ricci-flat.
\end{thm}

\begin{proof}
Suppose $\omega$ is a stationary point. By (\ref{id1}), we have
\[0=\dot\omega\wedge\frac{\omega^{n-1}}{(n-1)!}=-\frac{\|\Omega\|_\omega^{\frac{n-3}{n-1}}}{n-1}\left(R+\frac{2n-6}{n-2}\|\tau\|^2+\frac{1}{n-2}\|T\|^2\right)\frac{\omega^n}{n!}.\]
Therefore
\[2\Delta_c\log\|\Omega\|_\omega=R=-\frac{2n-6}{n-2}\|\tau\|^2-\frac{1}{n-2}\|T\|^2\leq 0,\]
where $\Delta_c$ is the complex Laplacian associated to $\omega$. By the maximum principle, we must have that $\log \| \Omega \|_\omega$ is a constant. It follows that $\|T\|^2=0$, and hence $\omega$ is K\"ahler and Ricci-flat.
\end{proof}

\section{Monotone Quantities and Related Estimates} \label{section-monotone}
In this section, we show that along the dual Anomaly flow, the norm $\| \Omega \|_\omega$ cannot shrink to zero. However, $\| \Omega \|_\omega$ may tend to infinity, and we will later study examples where this behavior is observed.
\par First, we introduce a functional which is monotone along the dual Anomaly flow.

\begin{thm}~\\
Suppose $\omega_t$ is a family of Hermitian metrics satisfying the dual Anomaly flow (\ref{daf}) with conformally balanced initial data. Then for any $\alpha\leq\dfrac{2}{n-1}<2$, the functional
\[\mathcal{F}_\alpha(\omega)=\int_X\|\Omega\|_\omega^{\alpha}\frac{\omega^n}{n!}\]
satisfies
\[
{d \over dt} \mathcal{F}_\alpha(\omega_t) \leq 0
\]
and is thus monotonically nonincreasing along the flow.
\end{thm}
\begin{proof}
Using (\ref{id1}) and (\ref{id2}), we have
\[\begin{split}\pt_t\left(\int_X \|\Omega\|_\omega^\alpha\frac{\omega^n}{n!}\right)&=\int_X \frac{\|\Omega\|_{\omega}^{\frac{n-3}{n-1}+\alpha}}{2(n-1)}(\alpha-2) \left(R+\frac{2n-6}{n-2}\|\tau\|^2+\frac{1}{n-2}\|T\|^2\right)\frac{\omega^n}{n!}\\ &\leq\frac{\alpha-2}{2(n-1)}\int_X \|\Omega\|_\omega^{\frac{n-3}{n-1}+\alpha}R\frac{\omega^n}{n!}\\ &=\frac{\alpha-2}{n-1}\int_X \|\Omega\|_\omega^{\alpha-\frac{2}{n-1}}i\pt\bpt\log\|\Omega\|_\omega\wedge\left(\|\Omega\|_\omega\frac{\omega^{n-1}}{(n-1)!}\right)\\ &=-\frac{\alpha-2}{n-1}\left(\alpha-\frac{2}{n-1}\right)\int_X \|\Omega\|_\omega^{\alpha+\frac{n-3}{n-1}}\|\nabla'\log\|\Omega\|_\omega\|^2\frac{\omega^n}{n!}\\ &\leq 0.\end{split}\]
\end{proof}

As an application, we prove the following lower bound for $\| \Omega \|_\omega$.

\begin{cor}
Suppose $\omega_t$ is a family of Hermitian metrics satisfying the dual Anomaly flow (\ref{daf}) on $X \times [0,T]$ with conformally balanced initial data. Then
\[
\| \Omega \|_\omega(x,t) \geq \inf_X \| \Omega \|_{\omega_0}
\]
for all $(x,t) \in X \times [0,T]$.
\end{cor}

\begin{proof}
The previous theorem shows that
\[
{d \over dt} \left( \int_X \| \Omega \|^{-p}_\omega \, {\omega^n \over n!} \right)^{1/p} \leq 0
\]
for all $p >0$. Therefore, at each $t \in [0,T]$, we have the inequality
\[
\left( \int_X \| \Omega \|^{-p}_\omega \, {\omega^n \over n!} \right)^{1/p} (t) \leq \left( \int_X \| \Omega \|^{-p}_{\omega_0} \, {\omega_0^n \over n!} \right)^{1/p}
\]
for all $p>0$. Taking the limit as $p \rightarrow \infty$, we obtain the estimate.
\end{proof}

We note that this argument can be adapted to the usual Anomaly flow. In particular, we have a monotone functional.

\begin{thm}~\\
Suppose $\omega_t$ is a family of Hermitian metrics satisfying the Anomaly flow with zero slope (\ref{af}) with conformally balanced initial data. Then for any $\alpha>2$, the functional $\mathcal{F}_\alpha$ satisfies
\[
{d \over dt} \mathcal{F}_\alpha(\omega_t) \leq 0
\]
along the flow.
\end{thm}
\begin{proof}
The evolution of Anomaly flow with zero slope is \cite{phong2018c}
\[\begin{split} \partial_t g_{\bar{k} j} &= \frac{1}{(n-1)\| \Omega \|_\omega} \bigg[ - \tilde{R}_{\bar{k} j} + {1 \over 2(n-2)} (\| T \|^2 - 2 \| \tau \|^2) g_{\bar{k} j}\\
& -\frac{1}{2} g^{q \bar{p}} g^{s \bar{r}} T_{\bar{k} q s} \bar{T}_{j \bar{p} \bar{r}} + g^{s \bar{r}} (T_{\bar{k} js} \bar{\tau}_{\bar{r}} + \tau_s \bar{T}_{j \bar{k} \bar{r}}) + \tau_j \bar{\tau}_{\bar{k}} ) \bigg],
\end{split}\]
when $n \geq 4$, and
\[
\partial_t g_{\bar{k} j} = \frac{1}{2 \| \Omega \|_\omega} \bigg[ - \tilde{R}_{\bar{k} j} + g^{m \bar{\ell}} g^{s \bar{r}} T_{\bar{r} mj} \bar{T}_{s \bar{\ell} \bar{k}} \bigg]
\]
when $n =3$. Taking the trace in either case gives
\[
\dot{\omega} \wedge \frac{\omega^{n-1}}{(n-1)!} = \frac{1}{(n-1)\| \Omega \|_\omega} \bigg[ -R + \frac{1}{(n-2)} \| T \|^2 + \frac{2(n-3)}{n-2} \| \tau \|^2 \bigg] \frac{\omega^n}{n!},
\]
and thus
\[
\partial_t \| \Omega \|_\omega = \frac{1}{2(n-1)} \bigg[ R - \frac{1}{(n-2)} \| T \|^2 - \frac{2(n-3)}{n-2} \| \tau \|^2 \bigg].
\]
Therefore
\[\begin{split}\pt_t\left(\int_X \|\Omega\|_\omega^\alpha\frac{\omega^n}{n!}\right)&=\int_X \frac{\|\Omega\|_{\omega}^{\alpha-1}}{2(n-1)}(\alpha-2) \left(R - \frac{1}{n-2} \| T \|^2 - \frac{2(n-3)}{n-2} \| \tau \|^2 \right)\frac{\omega^n}{n!}\\ &\leq\frac{\alpha-2}{2(n-1)}\int_X \|\Omega\|_\omega^{\alpha -1 }R\frac{\omega^n}{n!}\\ &=\frac{\alpha-2}{n-1}\int\|\Omega\|_\omega^{\alpha-2}i\pt\bpt\log\|\Omega\|_\omega\wedge\left(\|\Omega\|_\omega\frac{\omega^{n-1}}{(n-1)!}\right)\\ &=-\frac{(\alpha-2)^2}{n-1}\int\|\Omega\|_\omega^{\alpha-1}\|\nabla'\log\|\Omega\|_\omega\|^2\frac{\omega^n}{n!}\\ &\leq 0.\end{split}\]
\end{proof}

As in the case of the dual Anomaly flow, this monotone quantity gives the following estimate for the dilaton $\| \Omega \|_\omega$.

\begin{cor}
Suppose $\omega_t$ is a family of Hermitian metrics satisfying the Anomaly flow with zero slope (\ref{af}) on $X \times [0,T]$ with conformally balanced initial data. Then
\[
\| \Omega \|_\omega(x,t) \leq \sup_X \| \Omega \|_{\omega_0}
\]
for all $(x,t) \in X \times [0,T]$.
\end{cor}

\section{Examples} \label{section-examples}

\subsection{Conformally K\"ahler metrics} \label{section-conf-kahler}
As our first class of examples, we consider K\"ahler $n$-folds $(X^n,\Omega)$ and study the dual Anomaly flow for conformally K\"ahler initial data. In this case, the initial data satisfies $\omega_0 = e^f \chi$ for some K\"ahler metric $\chi$ and smooth function $f: X \rightarrow \mathbb{R}$. Combined with the conformally balanced condition (\ref{confbal}), this implies
\[
e^{(n-2)f} = C_0 \| \Omega \|^{-2}_\chi,
\]
for some constant $C_0$, and furthermore
\[
\| \Omega \|_{\omega(0)} \omega(0)^{n-1} = \alpha_0^{n-1},
\]
where $\alpha_0 = C_0^{1/2(n-1)} \chi$ is a K\"ahler metric. We study the ansatz
\[
\| \Omega \|_{\omega} \omega^{n-1} = \alpha_u^{n-1}, \ \ \ \alpha_u = \alpha_0 + i \partial \bar{\partial} u > 0,
\]
parameterized by a potential function $u(x,t)$. Solving for $\omega$ in terms of $\alpha_u$ (see e.g. \cite{phong2018c}), this ansatz can be expressed as
\[
\omega = \| \Omega \|_{\alpha_u}^{-2/(n-2)} \alpha_u,
\]
with
\[
\| \Omega \|_{\omega} = (\| \Omega \|_{\alpha_u})^{2(n-1)/(n-2)}.
\]
Substituting this ansatz into the dual Anomaly flow gives
\[
(n-1) \dot{\alpha}_{u} \wedge \alpha_{u}^{n-2} = - i \partial \bar{\partial} (\| \Omega \|_{\alpha_u}^{4} \| \Omega \|_{\alpha_u}^{-2} \alpha_u^{n-2} ),
\]
which reduces to
\[
i \partial \bar{\partial} \left[ (n-1) \dot{u} + \| \Omega \|_{\alpha_u}^2 \right] \wedge \alpha_u^{n-2} =0.
\]
Therefore a solution to the dual Anomaly flow can be produced by solving
\[
{d \over dt} u = - {1 \over (n-1)} \| \Omega \|_{\alpha_u}^2.
\]
This flow has been studied by Cao-Keller \cite{cao2013}, where it was named the $\Omega$-K\"ahler flow. We may also write this flow in the notation of Collins-Hisamoto-Takahashi \cite{collins2017}. Define $\rho_0 = \log \| \Omega \|_{\alpha_0}^2 - \log (n-1)$. Then the equation for $u(x,t)$ becomes the following inverse Monge-Amp\`ere flow
\[
{d \over dt} u = - {\alpha_0^n \over \alpha_u^n} e^{\rho_0}.
\]
Let $V = \int_X \alpha_0^n$ and consider
\[
v = u - E(u),
\]
where $E$ is the Aubin-Yau energy functional, defined by the variational formula
\[
\delta E(u) = {1 \over V} \int_X (\delta u) \, \alpha_u^n.
\]
Then $v$ satisfies
\[
{d \over dt} v = {1 \over V} \int_X e^{\rho_0} \alpha_0^n  - {\alpha_0^n \over \alpha_v^n} e^{\rho_0}.
\]
Let
\[
c = - \log {1 \over V} \int_X e^{\rho_0} \alpha_0^n,
\]
and reparameterize time such that
\[
\varphi(x,t) = v(x, e^c t).
\]
Then $\varphi$ satisfies
\begin{equation} \label{ma-1}
{d \over dt} \varphi = 1 - {\alpha_0^n \over \alpha_{\varphi}^n} e^{\rho_0 + c}.
\end{equation}
This is exactly the $MA^{-1}$-flow introduced by Collins-Hisamoto-Takahashi with $\lambda=0$. The $MA^{-1}$-flow is in part motivated by its interpretation as the gradient flow of the Ding functional. It is interesting that in this reduction, the dual Anomaly flow is a gradient flow. We do not know whether such a gradient flow interpretation exist for more general initial data.

By the results of \cite{cao2013,collins2017,fang2011}, the flow (\ref{ma-1}) for $\varphi$ exists for all time and converges to $\varphi_\infty$ such that $\alpha_{\varphi_\infty}$ is the unique K\"ahler Ricci-flat metric \cite{yau1978} in the class $[\alpha_0]$. Therefore the dual Anomaly flow with conformally K\"ahler initial data exists for all time and converges to a limiting K\"ahler Ricci-flat metric $\omega_\infty$ satisfying
\[
\| \Omega \|_{\omega_\infty} \omega_\infty^{n-1} = (\alpha_0 + i \partial \bar{\partial} \varphi_\infty)^{n-1}.
\]
In \cite{phong2018c}, the Anomaly flow (\ref{af}) is studied with conformally K\"ahler data, and though the ansatz reduces to a different PDE, the flow also exists for all time and converges to a K\"ahler Ricci-flat metric.

\subsection{Generalized Calabi-Gray manifolds} \label{section-gcg}

The Calabi-Gray manifold is a class of non-K\"ahler Calabi-Yau 3-folds first constructed by Calabi \cite{calabi1958} and Gray \cite{gray1969} using the vector cross product in 7-dimensional Euclidean space. In \cite{fei2016,fei2016b}, this construction was further generalized by the first-named author from the viewpoint of twistor spaces.

Roughly speaking, generalized Calabi-Gray manifolds are total spaces of holomorphic fibrations with hyperk\"ahler fibers over a Riemann surface $\Sigma$ equipped with a holomorphic map $\varphi$ to $\ccc\pp^1$ satisfying a certain condition. More details can be found in \cite{fei2017,fei2018b}.

As in \cite{fei2017}, we may take a special ansatz
\[\omega=e^{2f}\hat{\omega}+e^f\omega',\]
where $\hat{\omega}$ is a fixed canonical metric on $\Sigma$, $\omega'$ is a fiberwise Calabi-Yau metric on the hyperk\"ahler fibers and $f:\Sigma\to\rr$ is an arbitrary function on $\Sigma$.

By the calculation of \cite{fei2017}, we have
\[\|\Omega\|_\omega\omega^2=2e^f \hat{\omega} \wedge\omega'+\omega'^2\]
and
\[-i\pt\bpt\left(\|\Omega\|_\omega\omega\right)=-i\pt\bpt\left(e^{-f}\omega'\right) = - i \pt \bpt e^{-f} \wedge \omega' + \kappa e^{-f} \hat{\omega} \wedge \omega',\]
and therefore the dual Anomaly flow reduces to
\[\pt_t u=\frac{u^2}{4}(\Delta u-2\kappa u),\]
where $u=e^{-f}$, $\Delta$ and $\kappa\leq0$ are the Laplace operator and Gauss curvature associated to the metric $\hat{\omega}$.

This is a parabolic equation. By the maximum principle, given that $u_0=e^{-f_0}>0$, we know that $u$ stays positive. Moreover, it is easy to see that
\[-2\pt_t\int u^{-1}=\int(-\kappa)u\geq \frac{C}{\int u^{-1}}.\]
It follows that
\[\pt_t\left(\int u^{-1}\right)^2\leq -C,\]
therefore finite-time singularity always occur. Moreover, just as in \cite{fei2017d}, one can show that
\[\int u^{-1}(\alpha,\beta,\gamma):=V_0\in\rr^3\]
is a constant vector, where $\alpha,\beta,\gamma$ are components of the holomorphic map $\varphi:\Sigma\to\ccc\pp^1=S^2\subset\rr^3$ with $\alpha^2+\beta^2+\gamma^2=1$. Suppose $V_0\neq 0$, then we have the estimate
\[\int u^{-1}\geq \int u^{-1}(\alpha,\beta,\gamma)\cdot\frac{V_0}{|V_0|}=|V_0|>0.\]
This estimate shows that before $\int u^{-1}$ decreases to 0 singularities have already occurred.

\subsection{Iwasawa manifold}
Consider the quotient $X$ of $\mathbb{C}^3$ by the action
\[
(x,y,z) \mapsto (x+a, y+b, z + \bar{a} y + c),
\]
where $a,b,c \in \mathbb{Z}[i]$. The manifold $X$ is a Calabi-Yau threefold, as
\[
\Omega = dz \wedge dx \wedge dy,
\]
is a well-defined holomorphic nowhere vanishing $(3,0)$ form. We define $\theta \in \Omega^{1,0}(X)$ by
\[
\theta = dz - \bar{x} dy,
\]
which is also well-defined on $X$. The projection $(x,y,z) \mapsto (x,y)$ defines a map $\pi: X \rightarrow T^4$ with $T^2$ fibers. On the base $T^4$, there is the flat metric
\[
\hat{\omega} = i d x \wedge d \bar{x} + i dy \wedge d \bar{y},
\]
and for any function $u: T^4 \rightarrow \mathbb{R}$, we will consider the ansatz metric
\[
\omega = \pi^* (e^u \hat{\omega}) + i \theta \wedge \bar{\theta}.
\]
This is a special case of the Fu-Yau ansatz metric \cite{fu2008}. Let us check whether this ansatz is preserved by the dual Anomaly flow.

A direction computation gives
\[
\| \Omega \|_\omega = e^{-u}, \ \ \| \Omega \|_\omega \omega^2 = e^{u} \hat{\omega}^2 + 2 i \theta \wedge \bar{\theta} \wedge \hat{\omega}.
\]
The dual Anomaly flow becomes
\[
{d \over dt} (e^u \hat{\omega}^2) = - i \partial \bar{\partial} (\hat{\omega} + e^{-u} i \theta \wedge \bar{\theta}).
\]
This ansatz is not preserved unless $e^{u(t)}$ is a constant at each time. In this case,
\[
{d \over dt} e^{u(t)} \hat{\omega}^2 = - e^{-u(t)} {\hat{\omega}^2 \over 2}.
\]
This is the ODE
\[
\dot{u} = - {e^{-2u} \over 2},
\]
which has solution
\[
e^{u}  = (R - t)^{1/2}.
\]
Thus the dual Anomaly flow has solution
\[
\omega(t) = (R - t)^{1/2} \hat{\omega} + i \theta \wedge \bar{\theta},
\]
and the base $T^4$ shrinks to zero in finite time.

We note that unlike the Anomaly flow \cite{phong2018d}, a similar computation shows that the Fu-Yau ansatz is not preserved by the dual Anomaly flow for general Goldstein-Prokushkin threefolds \cite{goldstein2004}.

\subsection{Product fibrations}
Next, we consider the product $X = T^4 \times T^2$ with holomorphic volume form
\[
\Omega = dz \wedge dx \wedge dy,
\]
where $(x,y)$ are holomorphic coordinates on $T^4$ and $z$ is a holomorphic coordinate on $T^2$, and ansatz metric
\[
\omega =  e^{2\psi} \omega_{T^2} + e^{\psi} \omega',
\]
where $\psi: T^2 \rightarrow \mathbb{R}$ is a smooth function and
\[
\omega' = i dx \wedge d \bar{x} + i dy \wedge d \bar{y}, \ \ \ \omega_{T^2} = i dz \wedge d \bar{z}.
\]
As $\omega$ is conformally K\"ahler, this is a special case of our earlier discussion in \S \ref{section-conf-kahler}. We include this example because the metric ansatz is a simple instance of the ansatz introduced in \cite{fei2016,fei2016b} and discussed in \S \ref{section-gcg}. It is easy to derive the reduction of the dual Anomaly flow in this setup. Indeed,
\[
\| \Omega \|_\omega = e^{- 2 \psi}
\]
\[
\| \Omega \|_\omega \omega^2 = \omega'^2 + 2 e^{\psi} \omega' \wedge \omega_{T^2},
\]
and the dual Anomaly flow is
\[
{d \over dt} (\omega'^2 + 2 e^{\psi} \omega' \wedge \omega_{T^2}) = - i \partial \bar{\partial} (e^{-\psi} \omega' + \omega_{T^2}).
\]
The flow reduces to the Riemann surface $(T^2, \omega_{T^2})$ as the PDE
\[
\partial_t u =  {u^2 \over 4} \Delta u, \ \ \ u = e^{- \psi}.
\]
The maximum principle gives uniform bounds for $u$ in terms of the initial data
\[
\inf_{T^2} u(x,0) \leq u(x,t) \leq \sup_{T^2} u(x,0).
\]
As a parabolic equation with bounded coefficients, it follows \cite{krylov1980} that $C^{\delta,\delta/2}$ norms of $u$ are bounded, and therefore by Schauder estimates (e.g. \cite{krylov1996}), the $C^{2+\delta,1+\delta/2}$ norms of $u(x,t)$ are bounded along the flow. By a bootstrap argument, the flow exists smoothly for all time with all norms uniformly bounded. Furthermore, we have the conservation law
\begin{equation} \label{product-conservation}
{d \over dt} \int_{T^2} e^{\psi} \omega_{T^2} = {d \over dt} \int_{T^2} {1 \over u} \omega_{T^2}= 0.
\end{equation}
Stationary points of the flow satisfy $\Delta u_\infty =0$ and hence are constant functions. The Dirichlet energy
\[
I(u(t)) = {1 \over 2} \int_{T^2} i \partial u \wedge \bar{\partial} u,
\]
evolves by
\[
{d \over dt} I(t) = - \int_{T^2} \dot{u} \, i \partial \bar{\partial} u = - 2 \int_{T^2} u^{-2} \dot{u}^2 \, \omega_{T^2}.
\]
Thus
\[
{d \over dt} I(t) \leq -{2 \over \sup u^2(x,0)} \int_{T^2} \dot{u}^2 \, \omega_{T^2}.
\]
From here, our uniform estimates on all norms of $u$ show that if $\int \dot{u}^2$ does not go to zero as $t \rightarrow \infty$, then $I(t)$ becomes unbounded, which contradicts our uniform bounds. A standard argument shows that $u(t)$ converges smoothly to a constant as $t \rightarrow \infty$. This constant is detected by the conservation law (\ref{product-conservation}).

\subsection{Quotients of $SL(2,{\bf C})$}
Let $X$ be a threefold admitting a global holomorphic frame $\{ e_1, e_2, e_3 \}$ of $T^{1,0} X$ such that $[e_i,e_j] = \epsilon_{ijk} e_k$, where $\epsilon_{ijk}$ is the Levi-Civita symbol. Such a manifold can be constructed by taking lattice quotients of the complex Lie group $SL(2,{\bf C})$. If $\{e^1,e^2,e^3 \}$ is the dual frame, then we set
\[
\Omega = e^1 \wedge e^2 \wedge e^3
\]
to be the holomorphic volume form, and
\[
\hat{\omega} = i e^1 \wedge \bar{e}^1 + i e^2 \wedge \bar{e}^2 + i e^3 \wedge \bar{e}^3
\]
to be the reference Hermitian metric. We will study the flow using the ansatz
\[
\omega(t) = \rho(t) \hat{\omega},
\]
where $\rho(t)>0$ is a positive constant. These metrics are in fact conformally balanced, and this ansatz was used in \cite{fei2015} to solve the Hull-Strominger system on complex Lie groups. We have
\[
\| \Omega \|_\omega = \rho^{-3/2},
\]
and furthermore, a computation using the structure constants $\epsilon_{ijk}$ gives
\[
i \partial \bar{\partial} \hat{\omega} = {1 \over 2} \hat{\omega}^2.
\]
Substituting this ansatz into the dual Anomaly flow leads to the ODE
\[
{d \over dt} (\rho^{1/2} \hat{\omega}^2) = - {1 \over 2} \rho^{-1/2} \hat{\omega}^2.
\]
This has solution $\rho(t)^{1/2} = (R - t)^{1/2}$, hence the dual Anomaly flow has solution
\[
\omega(t) = (R-t) \hat{\omega},
\]
and the metric goes to zero everywhere in finite-time.

\bibliographystyle{alpha}

\bibliography{C:/Users/Piojo/Dropbox/Documents/Source}

\bigskip
\noindent Department of Mathematics, Columbia University, New York, NY 10027, USA\\
tfei@math.columbia.edu

\bigskip
\noindent Department of Mathematics, Harvard University, Cambridge, MA 02138, USA\\
spicard@math.harvard.edu

\end{document}